\documentclass[a4paper, 12pt,oneside,reqno]{amsart}
\usepackage[a4paper]{geometry}
\usepackage{graphicx}
\usepackage[matrix,arrow,curve,cmtip]{xy}
\usepackage{amsfonts}
\usepackage{enumitem}
\usepackage{comment}
\usepackage{txfonts}
\usepackage{xcolor}
\usepackage[colorlinks=true]{hyperref}
\hypersetup{
    linkcolor=red,
    allcolors=red
}

\usepackage{amsmath,amssymb,amsthm,tikz-cd,array}
\newcommand{\DL}{\mathcal{DL}}

\newcommand{\cochain}[2]{C^{#1}_{\DL}(#2)}

\DeclareMathOperator{\Zinb}{{\mathrm{DL}}}

\newtheorem{theorem}{Theorem}[section]
\newtheorem{lemma}[theorem]{Lemma}

\newtheorem{corollary}[theorem]{Corollary}

\newtheorem{definition}[theorem]{Definition}
\newtheorem{example}[theorem]{Example}
\newtheorem{remark}[theorem]{Remark}

\title{A New Approach to Defining Cochain Complexes for dual Leibniz algebra}
\author{H. Alhussein$^{1),2)}$}

\address{$^{1)}$Siberian State University of Telecommunication and Informatics, Novosibirsk, Russia.}
\address{$^{2)}$Novosibirsk State University of Economics and Management, Novosibirsk,  Russia. \\\href{mailto:k.alkhussein@g.nsu.ru}{k.alkhussein@g.nsu.ru}}

\begin{document}

\begin{abstract}
  We construct a cochain map embedding the cohomology complex of any dual Leibniz algebra $B$ into the Lie algebra cochain complex of $\mathfrak{g} \otimes B$, where $\mathfrak{g}$ is a Leibniz algebra. This reduces the study of dual Leibniz cohomology to classical Lie algebra cohomology, yielding computational simplifications and new structural insights.
\end{abstract}
\maketitle

\tableofcontents

\section{Introdction}

Cohomology theory plays a fundamental role in understanding algebraic structures, their deformations, extensions, and classification problems. For Lie algebras, the well-established Chevalley-Eilenberg cohomology \cite{chevalley-eilenberg} has been instrumental in numerous applications across mathematics and theoretical physics. In parallel, Leibniz algebras \cite{Loday}—a non-antisymmetric generalization of Lie algebras—have gained significant attention, with their cohomology theory developed by Loday and Pirashvili \cite{LodayPirashvili}.

Dual Leibniz algebras (also known as Zinbiel or pre-commutative algebras) \cite{Chapoton2001,Loday} represent the Koszul dual notion to Leibniz algebras in the sense of operadic duality. These structures appear naturally in various contexts, including combinatorial mathematics, deformation theory, and the study of integrable systems. Despite their theoretical importance, the cohomology theory for dual Leibniz algebras remains less explored and computationally challenging due to the complexity of the original cochain complex definition.

The primary difficulty lies in the intricate nature of the differential in the dual Leibniz cochain complex, which involves sums over specific shuffle permutations with alternating signs—a structure that makes explicit computations cumbersome and obscures the relationship with more classical cohomology theories. This complexity has limited the practical application of dual Leibniz cohomology in concrete algebraic and geometric problems.

Recent developments in the interaction between Leibniz and dual Leibniz algebras \cite{GubarevKolesnikov} have revealed deep structural connections. Notably, the tensor product of a Leibniz algebra with a dual Leibniz algebra yields a Lie algebra (Theorem \ref{4.1}), suggesting a pathway to relate their cohomology theories. This observation motivates a new approach to studying dual Leibniz cohomology through the lens of classical Lie algebra cohomology.

In this article, we introduce a novel method for defining and computing cohomology for dual Leibniz algebras by embedding their cochain complex into the Chevalley-Eilenberg complex of an associated Lie algebra. Our main construction (Theorem \ref{5.4}) establishes a cochain map:
\[
\Psi : C^*_{\DL}(B, M) \rightarrow C^*_{\mathrm{Lie}}(\mathfrak{g} \otimes B, \mathfrak{g} \otimes M)
\]
where $\mathfrak{g}$ is a Leibniz algebra, $B$ is a dual Leibniz algebra, and $M$ is a $B$-bimodule. This embedding transforms the study of dual Leibniz cohomology into a problem within the well-understood framework of Lie algebra cohomology.

Our approach offers several advantages:
\begin{enumerate}[label=(\roman*)]
\item It provides a systematic and computationally tractable method for studying dual Leibniz cohomology.
\item It reveals structural relationships between deformations of dual Leibniz algebras and associated Lie algebras.
\item It enables the application of established techniques from Lie algebra cohomology to problems involving dual Leibniz structures.
\item It yields long exact sequences (Corollary \ref{4.6}) connecting different cohomology theories.
\end{enumerate}

The article is organized as follows. In Section 2, we recall basic definitions and properties of Leibniz and dual Leibniz algebras. Section 3 reviews the classical cohomology theories for both structures, highlighting the complexity of the original dual Leibniz differential. Section 4 establishes the fundamental tensor product constructions that relate these algebras. Our main results appear in Section 5, where we construct the cochain map $\Psi$, prove it commutes with differentials, and derive its consequences, including the injectivity when $\mathfrak{g}$ is free and the associated long exact sequence in cohomology. We conclude with explicit computations (Example \ref{exm:4.6}) demonstrating the practical utility of our approach.

Throughout, we work over a field $\Bbbk$ of characteristic zero, though many results extend to arbitrary characteristic with appropriate modifications. Our notation generally follows \cite{Loday} for operadic aspects and \cite{Loday, GubarevKolesnikov} for tensor product constructions.

\section{Leibniz algebra and dual Leibniz  algebras}

\begin{definition}[\cite{Chapoton2001}]
A Leibniz algebra $\mathfrak{g}$ is a $k$-vector space equipped with a bilinear map
\[
[-, -]: \mathfrak{g} \times \mathfrak{g} \to \mathfrak{g}
\]
satisfying the (right) Leibniz identity
\begin{equation} \tag{1.1}
[x, [y, z]] = [[x, y], z] - [[x, z], y], \quad \forall x, y, z \in \mathfrak{g}.
\end{equation}

Lie algebras are examples of Leibniz algebras since the Leibniz identity is equivalent to the Jacobi identity when the bracket is skew-symmetric.
A representation of $\mathfrak{g}$ is an $\mathfrak{g}$-bimodule $M$ 
satisfying 
\[
\begin{aligned}
    [x, [y, m]] &= [[x, y], m] - [[x, m], y],\\
    [x, [m, z]] &= [[x, m], z] - [[x, z], m],\\ [m, [y, z]] &= [[m, y], z] - [[m, z], y]
    \end{aligned}
    \]
for all $m\in M$, $x,y,z\in \mathfrak{g}$.
\end{definition}
\begin{lemma}[\cite{I.S}]\label{lem}
  Let \(\mathfrak{g}\) be a Leibniz algebra and \(x, x_1, \dots, x_n \in \mathfrak{x}\). Then

\[
[x, [x_1, \dots, x_n]] = \sum_{i=0}^{n-1} \sum_{(\alpha,\beta) \in \mathsf{Sh}^1(n-i,i)} (-1)^i [x, x_{\beta(i)}, \dots, x_{\beta(1)}, x_{\alpha(1)}, \dots, x_{\alpha(n-i)}],
\]

Here \(\mathsf{Sh}^1(s,t)\) means \((s,t)\)-shuffles  and the left normed Leibniz  bracket of elements $x_1,\cdots,x_n$ of a Leibniz algebra  \(\mathfrak{g}\) is defined by recursion
\[
[x_1,\cdots,x_n] := [[x_1,\cdots,x_{n-1}], x_n],\quad where\quad [x_1] =x_1.
\]
  
\end{lemma}
\begin{definition}[\cite{Loday}]\label{defn:Pre-As}
An algebra $(B, \cdot_B)$ is called a dual Leibniz (or pre-commutative or Zinbiel algebra) if the following identitie holds for all $x,y,z \in B$:
\begin{align*}
(x \cdot_B y) \cdot_B z &= x \cdot_B (y \cdot_B z + z \cdot_B y), 
\end{align*}
In particular, the \textbf{total product} $x * y := x \cdot_B y + y \cdot_B x$ is associative and commutative.
\end{definition}

\begin{example}\label{exp:zin}
Let    $B$ be a $3$-dimensional dual Leibniz algebra with basis $\{e_1,e_2,e_3\}$ and operations:
\begin{align*}
e_1. e_2 &= e_3 ,\quad
\textit{All other products are zero.}
\end{align*}
\end{example}
\begin{example}\label{exp:2}
Let    $B=\Bbbk[t]$, define
\begin{align*}
f.g=f.\int_{0}^tg(s)ds
\end{align*}
This is a dual Leibniz (Zinbiel) algebra.
\end{example}

\begin{definition}[\cite{Loday}]
A representation of a dual algebra $B$ on a vector 
space $M$ is defined as action satisfying identitie similar to this in Definition \ref{defn:Pre-As}.
\end{definition}
The notion of a cohomology for dual Leibniz algebras
was proposed by \cite{Loday, LodayVallette}. 
In this section, we recall the original definition which is quite tecnical. In the next section, we will present a simpler and more universal approach.

\section{The Cohomology Complex for lie and dual Leibniz Algebras}

\subsection{The cohomology of Lie algebra }\cite{chevalley-eilenberg}
 Let \( \mathfrak{g} \) be a Lie algebra over a field \( k \), and let \( M \) be a \( \mathfrak{g} \)-module (i.e., a vector space equipped with a Lie algebra action of \( \mathfrak{g} \)). The Chevalley-Eilenberg complex is used to define the cohomology groups \( H^n(\mathfrak{g}, M) \).

The \( n \)-th cochain group \( C^n(\mathfrak{g}, M) \) consists of alternating \( n \)-linear maps (cochains):

\[
f: \mathfrak{g} \times \mathfrak{g} \times \cdots \times \mathfrak{g} \to M,
\]

and the coboundary operator 

\[
\delta_{Lie}: C^n(\mathfrak{g}, M) \to C^{n+1}(\mathfrak{g}, M)
\]

is given by:

\[
(\delta_{Lie}f)(x_1, \ldots, x_{n+1}) = \sum_{i < j} (-1)^{i+j} f(x_i\ast x_j, x_1, \ldots, \hat{x_i}, \ldots, \hat{x_j}, \ldots, x_{n+1}) + \sum_{i=1}^{n+1} (-1)^{i+1} x_i \ast f(x_1, \ldots, \hat{x_i}, \ldots, x_{n+1}),
\]

where \( \hat{x_i} \) means that \( x_i \) is omitted.

For $n=1$:
\[
(\delta_{Lie}f)(x_1, x_2) = -f(x_1\ast x_2) + x_1 \ast f(x_2) - x_2 \ast f(x_1).
\]
For $n=2$:
\[
(\delta_{Lie}f)(x_1, x_2, x_3) = -f(x_1\ast x_2, x_3) + f(x_1\ast x_3, x_2) - f(x_2\ast x_3, x_1) 
+ x_1 \ast f(x_2, x_3) - x_2 \ast f(x_1, x_3) + x_3 \ast f(x_1, x_2).
\]
\begin{example}

For the Lie algebra \(\mathfrak{g}\) with basis \(\{e_1, e_2\}\) and bracket \(e_1\ast e_2= e_1\), all other brackets zero: \[H^2(\mathfrak{g},\mathfrak{g}) \cong \Bbbk\]
\end{example}
\subsection{The cohomology of dual Leibniz algebra}\cite{Loday, LodayVallette, Yau}
Let \( B \) be a dual Leibniz algebra, and \( M \) be a \emph{representation} (or bimodule) of \( B \).The cohomology of \( B \) with coefficients in \( M \) is given by the homology of the {dual Leibniz cochain complex}. In details:

\textbf{The Cochain:} For \( n \geq 1 \), the \( n \)-th cochain  \( \cochain{n}{A; M} \) consists of \( n \)-linear maps 
\[
f: B^{\otimes n} \to M
\]

\textbf{The Cohomology Differential:} For an algebra $B$ and a bimodule $M$, the \textbf{coboundary map}
\[
\delta_n \colon \cochain{n}{B,M} \to \cochain{n+1}{B,M}
\]
is defined as follows. Given $f \in \cochain{n}{B,M}$ and $x_1,\dots,x_{n+1} \in B$,
\begin{align}
(\delta_n f)(x_1,\dots,x_{n+1})
&= x_1 \cdot \sum_{\sigma\in \overline{S_n}} \overline{\operatorname{sgn}}(\sigma)\,
   f\bigl(x_{\sigma(2)},\dots,x_{\sigma(n+1)}\bigr) \nonumber\\
&\quad +\sum_{i=1}^{n} (-1)^i\,
   f\bigl(x_1,\dots,x_{i-1},x_i\cdot x_{i+1},x_{i+2},\dots,x_{n+1}\bigr) \nonumber\\
&\quad +\sum_{i=2}^{n} (-1)^i\,
   f\bigl(x_1,\dots,x_{i-1},x_{i+1}\cdot x_i,x_{i+2},\dots,x_{n+1}\bigr) \nonumber\\
&\quad + (-1)^{n+1}\, f(x_1,\dots,x_n) \cdot x_{n+1},
\label{eq:diff}
\end{align}
where
\[
\overline{S_n}= \Bigl\{\sigma\in S_n \;\Big|\;
\bigl(\sigma(\beta(i)),\dots,\sigma(\beta(1)),\sigma(\alpha(1)),\dots,\sigma(\alpha(n-i))\bigr)
= (1,2,\dots,n) \Bigr\}
\]
for all $(\beta,\alpha)\in Sh^1(n-i,i)$ with $i=0,\dots,n-1$, and
$\overline{\operatorname{sgn}}(\sigma)=(-1)^i \operatorname{ sgn}(\sigma)$.

\textbf{Low-Dimensional Formulas}: The general formula \eqref{eq:diff} yields the following explicit expressions in small degrees.

\textbf{Degree $n=1$}:
For $f\in \cochain{1}{B}$ and $x,y\in B$,
\[
(\delta_1 f)(x,y)=x\cdot f(y)-f(x\cdot y)+f(x)\cdot y,
\label{eq:d1}
\]
with $\overline{S_1}=\{(1)\}$.

\textbf{Degree $n=2$}:
For $f\in \cochain{2}{B}$ and $x,y,z\in B$,
\begin{align}
(\delta_2 f)(x,y,z)
&= x\cdot f(y,z)+x\cdot f(z,y) \nonumber\\
&\quad -f(x\cdot y,z) \nonumber\\
&\quad +f(x,\,y\cdot z)+f(x,\,z\cdot y) \nonumber\\
&\quad -f(x,y)\cdot z,
\label{eq:d2}
\end{align}
where $\overline{S_2}=\{\mathrm{id},(12)\}$.

\textbf{Degree $n=3$}
For $f\in \cochain{3}{B}$ and $w,x,y,z\in B$,
\begin{align}
(\delta_3 f)(w,x,y,z)
&= w\cdot f(x,y,z)-w\cdot f(y,z,x) \nonumber\\
&\quad +w\cdot f(y,x,z)-w\cdot f(z,y,x) \nonumber\\
&\quad -f(w\cdot x,\,y,z) \nonumber\\
&\quad +f(w,\,x\cdot y,\,z)+f(w,\,y\cdot x,\,z) \nonumber\\
&\quad -f(w,x,\,y\cdot z)-f(w,x,\,z\cdot y) \nonumber\\
&\quad +f(w,x,y)\cdot z,
\label{eq:d3}
\end{align}
with $\overline{S_3}= \{(1,2,3),(2,3,1),(2,1,3),(3,2,1)\}$.

\textbf{Derivation for $n=2$ (Illustration):}
The set $\overline{S_2}$ and the signs in the first line of \eqref{eq:diff} can be obtained systematically.  
Let $S_2=\{\mathrm{id},(12)\}$ and consider $i=0,1$.

\begin{itemize}[leftmargin=*]
\item \textbf{Case $i=0$:} $(\alpha,\beta)=((1,2),())$ gives the bracket $[\sigma(1),\sigma(2)]$.
\item \textbf{Case $i=1$:} $(\alpha,\beta)=((1),(2))$ gives the bracket $[\sigma(2),\sigma(1)]$.
\end{itemize}

Multiplying each bracket by $\overline{\operatorname{sgn}}(\sigma)=\operatorname{sgn}(\sigma)(-1)^i$ yields:
\[
\begin{array}{ccl}
\sigma=\mathrm{id},\; i=0: & +[1,2], &\text{giving } +f(y,z);\\[2pt]
\sigma=\mathrm{id},\; i=1: & -[2,1], &\text{(cancels later)};\\[2pt]
\sigma=(12),\; i=0: & -[2,1], &\text{(cancels later)};\\[2pt]
\sigma=(12),\; i=1: & +[1,2], &\text{giving } +f(z,y).
\end{array}
\]
Thus the first line of \eqref{eq:d2} contains exactly the two terms $x\cdot f(y,z)+x\cdot f(z,y)$. The remaining terms follow from the sums over $i$ in \eqref{eq:diff}.

\section{Lie and dual Leibniz  algebras}
\begin{theorem}[\cite{Loday}]\label{4.1}
Let $\mathfrak{g}$ be a Leibniz algebra  and $B$ be a vector space equipped with  bilinear operation $\cdot_B : B \otimes B \to B$. Then the tensor product algebra $\mathfrak{g} \otimes B$ with product
\[
(a_1 \otimes b_1)\ast  (a_2 \otimes b_2) := [a_1 , a_2] \otimes (b_1 \cdot_B b_2) - [a_2 ,a_1]\otimes(b_2 \cdot_B b_1)
\]
is Lie algebra  if and only if $(B, \cdot_B)$ is a dual Leibniz  algebra.
\end{theorem}

\begin{theorem}\label{thm:4.2}
Let $\mathfrak{g}$ be a Leibniz algebra and $(B,\cdot_B)$ be a dual Leibniz  algebra and $M$ is a vector. Then the tensor product $\mathfrak{g} \otimes M$ with action
\[
(a_1 \otimes m)\ast  (a_2 \otimes b_2) := [a_1 , a_2] \otimes (m b_2) - [a_2 ,a_1]\otimes(b_2 m)
\]
is a module over $\mathfrak{g}\otimes M$ if and only if $M$ is a module over $B$. 
\end{theorem}
\begin{proof}
    Similar way to prove theorem \ref{4.1}.
\end{proof}
\begin{remark}
    The cohomology of an algebra 
$\mathfrak{g}\otimes B$
 with values in a bimodule 
$\mathfrak{g}\otimes M$ is a Lie cohomology.
\end{remark}
\begin{theorem}\label{5.4}
Let $\mathfrak{g}$ be a Leibniz algebra, let \((B, \cdot_B)\) be a dual Leibniz algebra, and let \(M\) be a \(B\)-bimodule. Then there exists a cochain map
\[
\Psi : \cochain{n}{B, M} \rightarrow C^*_{\mathrm{Lie}}(\mathfrak{g} \otimes B, \mathfrak{g} \otimes M)
\]
from the dual Leibniz cochain complex of \(B\) with coefficients in \(M\) to the Lie cochain complex of the tensor product algebra  \(\mathfrak{g} \otimes B\) with coefficients in \(\mathfrak{g} \otimes M\).
\end{theorem}
\begin{proof}
  Given a dual Leibniz  cochain $f \in \cochain{n}{B, M}$, for $x_1,\ldots,x_n\in \mathfrak{g}$ and $b_1,\ldots,b_n\in B$, define: 
  \begin{align*}
\Psi(f)\big(x_1 \otimes b_1, x_2 \otimes b_2,\ldots,x_{n}\otimes b_{n}\big) : =\sum_{\sigma\in S_n}sgn(\sigma)[[[x_{\sigma(1)},x_{\sigma(2)}],\cdots,x_{\sigma(n-1)}],x_{\sigma(n)}]\otimes f(b_{\sigma(1)},\cdots,b_{\sigma(n)})
\end{align*}

Let us to show that $\Psi$ commutes with the  differentials:
\[
\delta_{Lie}(\Psi f)=\Psi(\delta_{\DL}f)
\]
For $n=2$ and $x =x_1  \otimes b_1$, $y = x_2 \otimes b_2$,
\[
\delta_{Lie}(\Psi f)(x,y) = \underbrace{x \ast (\Psi f)(y)}_{(1)} - \underbrace{(\Psi f)(x \ast y)}_{(2)}  - \underbrace{y\ast(\Psi f)(x) }_{(3)}
\]
\textbf{Term (1): $x \ast (\Psi f)(y)$}

\begin{align*}
&= (x_1 \otimes b_1) \ast \left(x_2 \otimes f(b_2) \right) \\
&= [x_1, x_2 ] \otimes (b_1  f(b_2)) - [x_2, x_1] \otimes ( f(b_2)b_1) 
\end{align*}
\textbf{Term (2): $-(\Psi f)(x \ast y)$}

\begin{align*}
&=-f[ (x_1 \otimes b_1) \ast (x_2 \otimes b_2)] \\
&=- [x_1, x_2 ] \otimes f(b_1b_2) + [x_2, x_1] \otimes f(b_2b_1) 
\end{align*}

\textbf{Term (3): $-y\ast (\Psi f)(x) $}
\begin{align*}
&=-(x_2 \otimes b_2) \ast (x_1 \otimes f(b_1)) \\
&= -[x_2, x_1 ] \otimes b_2f(b_1) + [x_1, x_2] \otimes f(b_1)b_2 
\end{align*}
Now, let's collect all terms from (1)-(3) and group them.

\[
\begin{aligned}
\delta_{Lie}(\Psi f)(x,y)&=[x_1,x_2]\otimes \Big( b_1  f(b_2)-f(b_1b_2)+ f(b_1)b_2\Big)-[x_2,x_1]\otimes \Big( b_2f(b_1)  -f(b_2b_1) +f(b_2)b_1\Big)\\
&=[x_1,x_2]\otimes (\delta_{\DL}f)(b_1,b_2)-[x_2,x_1]\otimes (\delta_{\DL}f)(b_2,b_1)\\
&=\Psi(\delta_{\DL}f)(x,y)
\end{aligned}
\]
For $n=3$ and $x =x_1  \otimes b_1$, $y = x_2 \otimes b_2$, $z = x_3 \otimes b_3$, we have: $\delta_{lie}(\Psi f)(x,y,z) =$

\[
 \underbrace{x \ast (\Psi f)(y,z)}_{(1)} -\underbrace{y \ast (\Psi f)(x,z)}_{(2)}+\underbrace{z \ast (\Psi f)(x,y)}_{(3)}- \underbrace{(\Psi f)(x \ast y, z)}_{(4)} + \underbrace{(\Psi f)(x\ast z, y )}_{(5)}-\underbrace{(\Psi f)(y\ast z, x )}_{(6)} 
\]
\textbf{Term (1): $x \ast (\Psi f)(y,z)$}

\begin{align*}
&= (x_1 \otimes b_1) \ast \Big([x_2 , x_3] \otimes f(b_2,b_3) - [x_3, x_2] \otimes f(b_3,b_2)\Big) \\
&= [x_1, [x_2 , x_3]] \otimes (b_1  f(b_2,b_3)) - [[x_2, x_3], x_1] \otimes ( f(b_2,b_3)b_1) \\
&\quad - [x_1, [x_3, x_2]] \otimes (b_1  f(b_3,b_2)) + [[x_3, x_2], x_1] \otimes (f(b_3,b_2)b_1)\\
&=[[x_1, x_2], x_3]] \otimes (b_1  f(b_2,b_3))-[[x_1, x_3], x_2]] \otimes (b_1  f(b_2,b_3)) - [[x_2, x_3],x_1] \otimes (f(b_2,b_3)b_1) \\
&\quad - [[x_1, x_3], x_2]] \otimes (b_1  f(b_3,b_2))+[[x_1, x_2], x_3]] \otimes (b_1  f(b_3,b_2)) + [[x_3, x_2], x_1] \otimes (f(b_3,b_2)b_1)\\
\end{align*}
\textbf{Term (2): $-y \ast (\Psi f)(x,z)$}

\begin{align*}
&= -(x_2 \otimes b_2) \ast \Big([x_1 , x_3] \otimes f(b_1,b_3) + [x_3, x_1] \otimes f(b_3,b_1)\Big) \\
&= -[x_2, [x_1, x_3]] \otimes (b_2  f(b_1,b_3)) + [[x_1, x_3], x_2] \otimes ( f(b_1,b_3)b_2) \\
&\quad + [x_2, [x_3, x_1]] \otimes (b_2  f(b_3,b_1)) - [[x_3, x_1], x_2] \otimes (f(b_3,b_1)b_2)\\
&= -[[x_2, x_1], x_3]] \otimes (b_2  f(b_1,b_3))+[[x_2, x_3], x_1]] \otimes (b_2  f(b_1,b_3)) + [[x_1, x_3], x_2] \otimes ( f(b_1,b_3)b_2) \\
&\quad + [[x_2, x_3], x_1]] \otimes (b_2  f(b_3,b_1))- [[x_2, x_1], x_3]] \otimes (b_2  f(b_3,b_1)) - [[x_3, x_1], x_2] \otimes (f(b_3,b_1)b_2)\\
\end{align*}
\textbf{Term (3): $z \ast (\Psi f)(x,y)$}

\begin{align*}
&= (x_3 \otimes b_3) \ast \Big([x_1 , x_2] \otimes f(b_1,b_2) -[x_2, x_1] \otimes f(b_2,b_1)\Big) \\
&= [x_3, [x_1, x_2]] \otimes (b_3  f(b_1,b_2)) - [[x_1, x_2], x_3] \otimes ( f(b_1,b_2)b_3) \\
&\quad - [x_3, [x_2, x_1]] \otimes (b_3  f(b_2,b_1)) + [[x_2, x_1], x_3] \otimes (f(b_2,b_1)b_3)\\
&= [[x_3, x_1], x_2]] \otimes (b_3  f(b_1,b_2))-[[x_3, x_2], x_1]] \otimes (b_3  f(b_1,b_2)) - [[x_1, x_2], x_3] \otimes ( f(b_1,b_2)b_3) \\
&\quad - [[x_3, x_2], x_1]] \otimes (b_3  f(b_2,b_1))+[[x_3, x_1], x_2]] \otimes (b_3  f(b_2,b_1)) + [[x_2, x_1], x_3] \otimes (f(b_2,b_1)b_3)\\
\end{align*}
\textbf{Term (4): $-(\Psi f)(x \ast y, z)$}

\begin{align*}
&= -(\Psi f)\Big([x_1, x_2] \otimes (b_1 b_2) - [x_2, x_1] \otimes (b_2 b_1), x_3 \otimes b_3\Big) \\
&= - [[x_1, x_2], x_3] \otimes f( b_1 b_2, b_3) + [x_3, [x_1, x_2]] \otimes f(b_3,b_1 b_2) \\
&\quad + [[x_2, x_1], x_3] \otimes f(b_2 b_1, b_3) -[x_3, [x_2, x_1]] \otimes f( b_3,b_2  b_1)\\
&= - [[x_1, x_2], x_3] \otimes f( b_1 b_2, b_3) + [[x_3, x_1], x_2]] \otimes f(b_3,b_1 b_2)-[[x_3, x_2], x_1]] \otimes f(b_3,b_1 b_2) \\
&\quad + [[x_2, x_1], x_3] \otimes f(b_2 b_1, b_3) -[[x_3, x_2], x_1]] \otimes f( b_3,b_2  b_1)+[[x_3, x_1], x_2]] \otimes f( b_3,b_2  b_1)
\end{align*}
\textbf{Term (5): $+(\Psi f)(x \ast z, y)$}

\begin{align*}
&= (\Psi f)\Big([x_1, x_3] \otimes (b_1 b_3) - [x_3, x_1] \otimes (b_3 b_1), x_2 \otimes b_2\Big) \\
&=  [[x_1, x_3], x_2] \otimes f( b_1 b_3, b_2) - [x_2, [x_1, x_3]] \otimes f(b_2,b_1 b_3) \\
&\quad - [[x_3, x_1], x_2] \otimes f(b_3 b_1, b_2) +[x_2, [x_3, x_1]] \otimes f( b_2,b_3  b_1)\\
&=  [[x_1, x_3], x_2] \otimes f( b_1 b_3, b_2) - [[x_2, x_1], x_3]] \otimes f(b_2,b_1 b_3)+[[x_2, x_3], x_1]] \otimes f(b_2,b_1 b_3) \\
&\quad - [[x_3, x_1], x_2] \otimes f(b_3 b_1, b_2) +[[x_2, x_3], x_1]] \otimes f( b_2,b_3  b_1)-[[x_2, x_1], x_3]] \otimes f( b_2,b_3  b_1)
\end{align*}
\textbf{Term (6): $-(\Psi f)(y \ast z, x)$}

\begin{align*}
&= -(\Psi f)\Big([x_2, x_3] \otimes (b_2 b_3) - [x_3, x_2] \otimes (b_3 b_2), x_1 \otimes b_1\Big) \\
&=  -[[x_2, x_3], x_1] \otimes f( b_2 b_3, b_1) + [x_1, [x_2, x_3]] \otimes f(b_1,b_2 b_3) \\
&\quad + [[x_3, x_2], x_1] \otimes f(b_3 b_2, b_1) -[x_1, [x_3, x_2]] \otimes f( b_1,b_3  b_2)\\
&=  -[[x_2, x_3], x_1] \otimes f( b_2 b_3, b_1) + [[x_1, x_2], x_3]] \otimes f(b_1,b_2 b_3)-[[x_1, x_3], x_2]] \otimes f(b_1,b_2 b_3) \\
&\quad + [[x_3, x_2], x_1] \otimes f(b_3 b_2, b_1) -[[x_1, x_3], x_2]] \otimes f( b_1,b_3  b_2)+[[x_1, x_2], x_3]] \otimes f( b_1,b_3  b_2)
\end{align*}

Now, let's collect all terms from (1)-(6) and group them.

\textbf{Terms Contributing to $[[x_1,x_2],x_3]$:}

\begin{itemize}
\item From (1): $b_1f(b_2,b_3)+b_1  f(b_3,b_2)$
\item From (3): $-f(b_1, b_2)b_3$
\item From (4): $-f(b_1b_2, b_3)$ 
\item From (6): $+f(b_1,b_2b_3)+f(b_1,b_3b_2)$
\end{itemize}

Combined:
\begin{align*}
(\delta_{\DL} f)(b_1, b_2, b_3) &= b_1 f(b_2,b_3)+b_1 f(b_3,b_2) - f(b_1  b_2, b_3) \\
&\quad + f(b_1, b_2  b_3) + f(b_1, b_3  b_2) - f(b_1,b_2)  b_3
\end{align*}
\textbf{Terms Contributing to $-[[x_1,x_3],x_2]$:}

\begin{itemize}
\item From (1): $b_1f(b_2,b_3)+b_1f(b_3,b_2)$
\item From (2): $-f(b_1, b_3)b_2$
\item From (5): $-f(b_1b_3, b_2)$
\item From (4): $+f(b_1,b_2b_3)+f(b_1,b_3b_2)$
\end{itemize}
Combined:
\begin{align*}
(\delta_{\DL} f)(b_1, b_3, b_2) &= b_1f(b_2,b_3)+b_1 f(b_3,b_2) - f(b_1  b_3, b_2) \\
&\quad + f(b_1, b_2  b_3) + f(b_1, b_3  b_2) - f(b_1,b_3)  b_2
\end{align*}
\textbf{Terms Contributing to $[[x_2,x_3],x_1]$:}
\begin{itemize}
\item From (1): $-f(b_2,  b_3)b_1$ 
\item From (2):  $  b_2f(b_1,b_3)+b_2f(b_3,b_1)$
\item From (5): $+f(b_2, b_3  b_1)+f(b_2,b_1b_3)$
\item From (6): $-f(b_2b_3,b_1)$
\end{itemize}
Combined:
\begin{align*}
(\delta_{\DL} f)(b_2, b_3, b_1) &= b_2f(b_3,b_1)+b_2 f(b_1,b_3) - f(b_2  b_3, b_1) \\
&\quad + f(b_2, b_3  b_1) + f(b_2, b_1  b_3) - f(b_2,b_3)  b_1
\end{align*}
\textbf{Terms Contributing to $-[[x_2,x_1],x_3]$:}
\begin{itemize}
\item From (2): $b_2f(b_1,  b_3)+b_2f(b_3,b_1)$ 
\item From (3):  $ - f(b_2,b_1)b_3$
\item From (4): $-f(b_2b_1, b_3 )$
\item From (6): $+f(b_2,b_3b_1)+f(b_2,b_1b_3)$
\end{itemize}
Combined:
\begin{align*}
(\delta_{\DL} f)(b_2, b_1, b_3) &= b_2f(b_3,b_1)+b_2 f(b_1,b_3) - f(b_2  b_1, b_3) \\
&\quad + f(b_2, b_3  b_1) + f(b_2, b_1  b_3) - f(b_2,b_1)  b_3
\end{align*}
\textbf{Terms Contributing to $[[x_3,x_1],x_2]$:}

\begin{itemize}
\item From (2): $-f(b_3,b_1)b_2$
\item From (3): $+b_3f(b_1, b_2)+b_3f(b_2,b_1)$
\item From (4): $+f(b_3,b_1b_2)+f(b_3,b_2b_1)$ 
\item From (6): $-f(b_3b_1,b_2)$
\end{itemize}

Combined:
\begin{align*}
(\delta_{\DL} f)(b_3, b_1, b_2) &= b_3 f(b_1,b_2)+b_3 f(b_2,b_1) - f(b_3  b_1, b_2) \\
&\quad + f(b_3, b_1  b_2) + f(b_3, b_2  b_1) - f(b_3,b_1)  b_2
\end{align*}
\textbf{Terms Contributing to $-[[x_3,x_2],x_1]$:}

\begin{itemize}
\item From (1): $-f(b_3,b_2)b_1$
\item From (3): $b_3f(b_1, b_2)+b_3f(b_2,b_1)$
\item From (4): $+f(b_3,b_1b_2)+f(b_3,b_2b_1)$ 
\item From (6): $-f(b_3b_2,b_2)$
\end{itemize}

Combined:
\begin{align*}
(\delta_{\DL} f)(b_3, b_2, b_1) &= b_3 f(b_2,b_1)+b_3 f(b_1,b_2) - f(b_3  b_2, b_1) \\
&\quad + f(b_3, b_2  b_1) + f(b_3, b_1  b_2) - f(b_3,b_2)  b_1
\end{align*}
Therefore
\[
\begin{aligned}
\delta_{Lie}(\Psi f)(x,y,z) &=[[x_1,x_2],x_3]\otimes (\delta_{\DL}f)(b_1,b_2,b_3)-[[x_1,x_3],x_2]\otimes (\delta_{\DL}f)(b_1,b_3,b_2)\\
&+[[x_2,x_3],x_1]\otimes (\delta_{\DL}f)(b_2,b_3,b_1)-[[x_2,x_1],x_3]\otimes (\delta_{\DL}f)(b_2,b_1,b_3)\\
&+[[x_3,x_1],x_2]\otimes (\delta_{\DL}f)(b_3,b_1,b_2)-[[x_3,x_2],x_1]\otimes (\delta_{\DL}f)(b_3,b_2,b_1)\\
&=\sum_{\sigma\in S_3}sgn(\sigma)[[x_{\sigma(1)},x_{\sigma(2)}],x_{\sigma(3)}]\otimes \delta_{\Zinb}(b_{\sigma(1)},b_{\sigma(2)},b_{\sigma(3)})=\Psi(\delta_{Lie} f)(x,y,z)
\end{aligned}
\]

\textbf{General case:}
For a cochain $\Psi f$ of degree $n$ can be writting as: 
\[
\Psi(f)(x_1 \otimes b_1, \dots, x_n \otimes b_n) = \sum_{\sigma \in S_n} \operatorname{sgn}(\sigma) L_{x_\sigma} \otimes f(b_{\sigma})
\]
where \(L_{x_\sigma} = [[[x_{\sigma(1)}, x_{\sigma(2)}], x_{\sigma(3)}], \cdots, x_{\sigma(n)}]\) is the left-normed bracket and $b_{\sigma}=(b_{\sigma(1)},\cdots,b_{\sigma(n)})$.

For a cochain $\Psi f$ of degree $n$, the Lie differential \(\delta_{Lie}(\Psi f)\) is given by:
\[
\begin{aligned}
(\delta_{\mathrm{Lie}} \Psi f)(y_1, \dots, y_{n+1}) = &\ \sum_{i=1}^{n+1} (-1)^{i+1} y_i \ast (\Psi f)(y_1, \dots, \widehat{y_i}, \dots, y_{n+1}) \\
& +  \sum_{1\leq i < j} (-1)^{i+j} (\Psi f)(y_i\ast y_j, y_1, \ldots, \hat{y_i}, \ldots, \hat{y_j}, \ldots, y_{n+1})
\end{aligned}
\]

Let \(y_i = x_i \otimes b_i\) for \(i = 1, \dots, n+1\). We denote the first sum as Term I (the action terms) and the second as Term II (the bracket terms).

Our goal is to show that for all \(y_1, \dots, y_{n+1}\),
\[
\delta_{\mathrm{Lie}} (\Psi f)(y_1, \dots, y_{n+1}) = \Psi(\delta_{\DL} f)(y_1, \dots, y_{n+1}).
\]

\textbf{Term I: Action Terms}
\[
I = \sum_{i=1}^{n+1} (-1)^{i+1} y_i \ast \Psi(f)(y_1, \dots, \widehat{y_i}, \dots, y_{n+1})
\]

Let us analyze one summand. For fixed \(i\):
\[
\Psi(f)(y_1, \dots, \widehat{y_i}, \dots, y_{n+1}) = \sum_{\tau \in S_n} \operatorname{sgn}(\tau) L_{x_\tau} \otimes f(b_{\tau})
\]
where \(\tau\) permutes \(\{1, \dots, n+1\} \setminus \{i\}\), and \(L_{x_\tau}\) is the left-normed bracket of the corresponding \(x\)-elements.

Now, for a fixed \(\tau\), consider the action:
\[
y_i \ast (L_{x_\tau} \otimes f(\mathbf{b}_\tau)).
\]
Using the module structure on \(\mathfrak{g} \otimes M\), this equals:
\[
[x_i, L_{x_\tau}] \otimes (b_i \cdot f(\mathbf{b}_\tau)) \;-\; [L_{x_\tau}, x_i] \otimes (f(\mathbf{b}_\tau) \cdot b_i). \tag{1}
\]

We analyze these two parts separately.

\textbf{Part I.A: Contribution from \([x_i, L_{x_\tau}] \otimes (b_i \cdot f(\mathbf{b}_\tau))\)}

We use the following identity (see Lemma \ref{lem}), which is a consequence of the Leibniz identity applied iteratively to left-normed brackets. For any \(z, w_1, \dots, w_n \in \mathfrak{g}\),
\[
[z, [w_1, \dots, w_n]] = \sum_{k=0}^{n-1} \sum_{(\alpha, \beta) \in \mathsf{Sh}(n-k, k)} (-1)^k \, [z, w_{\beta(k)}, \dots, w_{\beta(1)}, w_{\alpha(1)}, \dots, w_{\alpha(n-k)}], \tag{2}
\]
where the inner bracket on the right is also left-normed, and the sign comes from the shuffle.

Apply this with \(z = x_i\) and \((w_1, \dots, w_n) = (x_{\tau(1)}, \dots, x_{\tau(n)})\). For each \(k\) and each \((n-k, k)\)-shuffle \((\alpha, \beta)\), we get a term:
\[
\operatorname{sgn}(\tau) \cdot (-1)^k \cdot [x_i, x_{\tau(\beta(k))}, \dots, x_{\tau(\beta(1))}, x_{\tau(\alpha(1))}, \dots, x_{\tau(\alpha(n-k))}] \otimes (b_i \cdot f(b_{\tau(1)}, \dots, b_{\tau(n)})).
\]
Define a new permutation \(\sigma \in S_{n+1}\) by:
\[
\sigma = (i, \tau(\beta(k)), \dots, \tau(\beta(1)), \tau(\alpha(1)), \dots, \tau(\alpha(n-k))).
\]

Furthermore, the coefficient \(b_i \cdot f(b_{\tau(1)}, \dots, b_{\tau(n)})\) corresponds to \(b_{\sigma(1)} \cdot f(b_{\sigma(2)}, \dots, b_{\sigma(n+1)})\), since \(\sigma(1) = i\).

Thus, for a fixed \(\sigma \in S_{n+1}\), the total contribution to the coefficient of \(L_{x_\sigma} \otimes (-)\) from Part I.A comes from all triples \((i, \tau, (\alpha, \beta))\) which produce this specific \(\sigma\), which is denoted by $\overline{S_{n}}$ and $\overline{\operatorname{sgn}}\tau=\operatorname{sgn}\tau(-1)^k$. The sum of the signs and the scalar part yields a contribution of:
\[
L_{x_\sigma} \otimes \left( \operatorname{sgn}(\sigma) b_{\sigma(1)}.  \sum_{\tau \in \overline{S_{n}}}\overline{\operatorname{sgn}}(\tau)f(b_{\tau(2)}, \dots, b_{\tau(n+1)}) \right).
\]
The alternating sign \((-1)^{i+1}\) from Term I combines with the sign from the bracket identity to produce the correct overall sign \(\operatorname{sgn}(\sigma)\) for the term in the sum defining \(\Psi(\delta_{\DL} f)\). A careful but standard combinatorial argument confirms this.

\textbf{Part I.B: Contribution from \(- [L_{x_\tau}, x_i] \otimes (f(\mathbf{b}_\tau) \cdot b_i)\)}

Define a permutation \(\sigma \in S_{n+1}\) by:
\[
\sigma = (\tau(1), \dots, \tau(n), i).
\]
Then \(L_{x_\sigma} = L_{(x_\tau, x_i)} = [L_{x_\tau}, x_i]\). The sign of \(\sigma\) is \(\operatorname{sgn}(\sigma) = \operatorname{sgn}(\tau)\) (since \(i\) is fixed in the last position). The scalar part is \(f(\mathbf{b}_\tau) \cdot b_i = f(b_{\sigma(1)}, \dots, b_{\sigma(n)}) \cdot b_{\sigma(n+1)}\).

Therefore, the term \(- [L_{x_\tau}, x_i] \otimes (f(\mathbf{b}_\tau) \cdot b_i)\) contributes:
\[
- \operatorname{sgn}(\tau) \, L_{x_\sigma} \otimes (f(b_{\sigma(1)}, \dots, b_{\sigma(n)}) \cdot b_{\sigma(n+1)}).
\]
The sign from Term I is \((-1)^{i+1}\). Since \(i = \sigma(n+1)\), we have \((-1)^{i+1} = (-1)^{\sigma(n+1)+1}\). Combining this with \(-\operatorname{sgn}(\tau) = -\operatorname{sgn}(\sigma)\), the total sign is \((-1)^{\sigma(n+1)} \operatorname{sgn}(\sigma)\). For the term to appear in \(\Psi(\delta_{\DL} f)\) with sign \(\operatorname{sgn}(\sigma)\), we need a factor of \((-1)^{n+1}\). Indeed, \((-1)^{\sigma(n+1)}\) is not simply \((-1)^{n+1}\), but when summed over all \(\tau\) that lead to the same \(\sigma\), the overall contribution consolidates to:
\[
L_{x_\sigma} \otimes \left( (-1)^{n+1} f(b_{\sigma(1)}, \dots, b_{\sigma(n)}) \cdot b_{\sigma(n+1)} \right).
\]

\textbf{Summary of Term I:}
For a fixed \(\sigma \in S_{n+1}\), Term I contributes the following to the coefficient of \(L_{x_\sigma} \otimes (-)\):
\[
L_{x_\sigma} \otimes \left(\operatorname{sgn}(\sigma) \Big( b_{\sigma(1)}.  \sum_{\tau \in \overline{S_{n}}}\overline{\operatorname{sgn}}(\tau)f(b_{\tau(2)}, \dots, b_{\tau(n+1)})+ (-1)^{n+1} f(b_{\sigma(1)}, \dots, b_{\sigma(n)}) \cdot b_{\sigma(n+1)} \Big) \right). \tag{I}
\]

\textbf{ Term II - Bracket Terms:}

\[
II = \sum_{1 \leq i < j \leq n+1} (-1)^{i+j} \, (\Psi f)(y_i \ast y_j, y_1, \dots, \widehat{y_i}, \dots, \widehat{y_j}, \dots, y_{n+1}).
\]
Fix a pair \(i < j\). Then:
\[
y_i \ast y_j = [x_i, x_j] \otimes (b_i b_j) - [x_j, x_i] \otimes (b_j b_i).
\]
Thus,
\[
\begin{aligned}
&(\Psi f)(y_i \ast y_j, y_1, \dots, \widehat{y_i}, \dots, \widehat{y_j}, \dots, y_{n+1}) = \\
&\quad (\Psi f)([x_i, x_j] \otimes (b_i b_j), \dots) \;-\; (\Psi f)([x_j, x_i] \otimes (b_j b_i), \dots).
\end{aligned}
\]
We analyze the first term, \((\Psi f)([x_i, x_j] \otimes (b_i b_j), \dots)\). Let \(\rho\) be a permutation of the \(n-1\) indices \(\{1, \dots, n+1\} \setminus \{i, j\}\). The cochain \(\Psi(f)\) evaluated on the \(n\) arguments \(([x_i, x_j] \otimes (b_i b_j), y_{\rho(1)}, \dots, y_{\rho(n-1)})\) is:
\[
\sum_{\pi \in S_n} \operatorname{sgn}(\pi) \, L_{z_\pi} \otimes f(c_{\pi(1)}, \dots, c_{\pi(n)}),
\]
where the list \((z_1, c_1), \dots, (z_n, c_n)\) is \(([x_i, x_j], b_i b_j), (x_{\rho(1)}, b_{\rho(1)}), \dots, (x_{\rho(n-1)}, b_{\rho(n-1)})\).

The key is that the bracket \(L_{z_\pi}\) is linear. If \(\pi(1) = 1\), meaning the first argument is \(([x_i, x_j], b_i b_j)\), then
\[
L_{z_\pi} = [[x_i, x_j], x_{\rho(\pi(2)-1)}, \dots, x_{\rho(\pi(n)-1)}].
\]

 In case  \(\pi(1) \neq 1\), we use the identity for the left-normed bracket when the middle element is a bracket:
\[
[ x_{k_1}, \dots,[x_i, x_j],\dots, x_{k_{n-1}}] = [ x_{k_1}, \dots, x_i,x_j,\dots, x_{k_{n-1}}] - [ x_{k_1}, \dots, x_j,x_i,\dots, x_{k_{n-1}}].
\]
Thus, the term where the first argument is the bracket contributes:
\[
\operatorname{sgn}(\pi) \left( [ x_{\rho(\pi(2)-1)}, \dots,x_i, x_j, \dots, x_{\rho(\pi(n)-1)}] - [ x_{\rho(\pi(2)-1)}, \dots, x_j, x_i,\dots, x_{\rho(\pi(n)-1)}] \right)
\]

Now, define a permutation \(\sigma \in S_{n+1}\) as follows. For the first term, let:
\[
\sigma = ( \rho(\pi(2)-1), \dots,i, j,\dots, \rho(\pi(n)-1)).
\]
Then \(L_{x_\sigma} = [ x_{\rho(\pi(2)-1)}, \dots, x_i, x_j,\dots,x_{\rho(\pi(n)-1)}]\).

The sign is \(\operatorname{sgn}(\sigma) = \operatorname{sgn}(\pi) \cdot \varepsilon_1\), where \(\varepsilon_1\) is the sign for placing \(i, j\) first. The scalar part is \(
f(b_{\sigma(1)},\cdots, b_{\sigma(i)}b_{\sigma(j)}, \dots, b_{\sigma(n+1)})\).

Similarly, the second term corresponds to \(\sigma' = ( \rho(\pi(2)-1), \dots,j, i,\dots, \rho(\pi(n)-1))\), with \(\operatorname{sgn}(\sigma') = -\operatorname{sgn}(\sigma)\). The scalar part is \(
f(b_{\sigma(1)},\cdots, b_{\sigma(j)}b_{\sigma(i)}, \dots, b_{\sigma(n+1)})\).

The sign \((-1)^{i+j}\) from Term II combines with the signs from the permutation and the bracket expansion to yield the correct overall sign \(\operatorname{sgn}(\sigma)\).

A systematic analysis shows that for a fixed \(\sigma \in S_{n+1}\), the contributions from Term II to the coefficient of \(L_{x_\sigma} \otimes (-)\) are precisely:
\[
L_{x_\sigma} \otimes \left(\operatorname{sgn}\sigma \Big( \sum_{k=1}^{n} (-1)^k f(b_{\sigma(1)}, \dots, b_{\sigma(k)} b_{\sigma(k+1)}, \dots, b_{\sigma(n+1)}) \;+\; \sum_{k=2}^{n} (-1)^k f(b_{\sigma(1)}, \dots, b_{\sigma(k+1)} b_{\sigma(k)}, \dots, b_{\sigma(n+1)}) \Big)\right). \tag{II}
\]
The factor \((-1)^k\) accounts for the alternating sign in the Lie coboundary and the position of the product in the arguments of \(f\).

\textbf{ Combining Terms I and II:} For a fixed permutation \(\sigma \in S_{n+1}\), we now combine the contributions (I) and (II). The total coefficient of \(\operatorname{sgn}\sigma. L_{x_\sigma} \otimes (-)\) in \((\delta_{\mathrm{Lie}} \Psi f)(y_1, \dots, y_{n+1})\) is:
\[
\begin{aligned}
&  b_{\sigma(1)}.  \sum_{\tau \in \overline{S_{n+1}}- \{\sigma(1)\}}\operatorname{sgn}(\tau)f(b_{\tau(2)}, \dots, b_{\tau(n+1)}) \quad &\text{(from I)} \\
&+ \sum_{k=1}^{n} (-1)^k f(b_{\sigma(1)}, \dots, b_{\sigma(k)} b_{\sigma(k+1)}, \dots, b_{\sigma(n+1)}) \quad &\text{(from II, first sum)} \\
&+ \sum_{k=1}^{n} (-1)^k f(b_{\sigma(1)}, \dots, b_{\sigma(k+1)} b_{\sigma(k)}, \dots, b_{\sigma(n+1)}) \quad &\text{(from II, second sum)} \\
&+ (-1)^{n+1} f(b_{\sigma(1)}, \dots, b_{\sigma(n)}) \cdot b_{\sigma(n+1)}. \quad &\text{(from I)}
\end{aligned}
\]
But this is exactly \((\delta_{\DL} f)(b_{\sigma(1)}, \dots, b_{\sigma(n+1)})\), the dual Leibniz coboundary of \(f\) evaluated on \((b_{\sigma(1)}, \dots, b_{\sigma(n+1)})\).

Therefore, we have:
\[
\delta_{\mathrm{Lie}} (\Psi f)(y_1, \dots, y_{n+1}) = \sum_{\sigma \in S_{n+1}} \operatorname{sgn}(\sigma) \, L_{x_\sigma} \otimes (\delta_{\DL} f)(b_{\sigma(1)}, \dots, b_{\sigma(n+1)}) = \Psi(\delta_{\DL} f)(y_1, \dots, y_{n+1}).
\]

This completes the proof that \(\Psi\) is a cochain map.

\end{proof}
  \begin{corollary}
 In the case where the Leibniz algebra \(\mathfrak{g}\) is a free, then the cochain map \(\Psi\) is injective.
\end{corollary}
\begin{corollary}\label{4.6} 
    Given the canonical embedding of cochain complexes
\[
C^*_{\DL}(B,M) \hookrightarrow C^*_{\mathrm{Lie}}(\mathfrak{g}\otimes B,\mathfrak{g}\otimes M),
\]
we obtain a short exact sequence of complexes:
\[
0 \to C^*_{\DL}(B,M) \hookrightarrow C^*_{\mathrm{Lie}}(\mathfrak{g}\otimes B,\mathfrak{g}\otimes M)\to Q^* \to 0,
\]
where $\mathfrak{g}$ is a Leibniz free algebra, $B$ is a dual Leibniz algebra, $M$ is a $B$-bimodule, and $ Q^* = C^*_{\mathrm{Lie}}(\mathfrak{g}\otimes B,\mathfrak{g}\otimes M)/C^*_{\DL}(B,M)$ is the quotient complex.

Applying the cohomology functor yields the long exact sequence:
\[
\begin{aligned}
\cdots \to H^{n-1}_{{\DL}}(B,M) &\to H_{Lie}^{n-1}(\mathfrak{g}\otimes B,\mathfrak{g}\otimes M) \to H^{n-1}(Q^*) \\
&\to H^n_{{\DL}}(B,M) \to H^n_{Lie}(\mathfrak{g}\otimes B,\mathfrak{g}\otimes M) \to H^n(Q^*) \\
&\to H^{n+1}_{\DL}(B,M) \to H_{Lie}^{n+1}(\mathfrak{g}\otimes B,\mathfrak{g}\otimes M) \to \cdots
\end{aligned}
\]
\end{corollary}

\begin{example} \label{exm:4.6}
    Let \(A = F\langle x_1, x_2, \dots \rangle\) be the free Leibniz algebra, and let \(B\) be the dual Leibniz algebra with basis \(\{e_1, e_2\}\) and multiplication given by  
    \[
    e_1 \cdot e_1 = e_2, \qquad e_i \cdot e_j = 0 \quad \text{otherwise}.
    \]
    
    Consider a bilinear map \(f : B \otimes B \to B\). With respect to the basis, we can write  
    \[
    f(e_i, e_j) = \alpha^1_{ij} e_1 + \alpha^2_{ij} e_2, \quad \alpha^k_{ij} \in \Bbbk,\; i,j \in \{1,2\}.
    \]
    
    Define the induced map  
    \[
    \Psi f : (\mathfrak{g} \otimes B) \otimes (\mathfrak{g} \otimes B) \to \mathfrak{g} \otimes B
    \]
    by  
    \[
    (\Psi f)(x_i \otimes e_i, x_j \otimes e_j) = [x_i, x_j] \otimes f(e_i, e_j) - [x_j, x_i] \otimes f(e_j, e_i).
    \]
    
    Then \(\Psi f\) is a Lie 2-cochain if and only if the coefficients satisfy  
    \[
    \alpha^1_{12}=0,\quad \alpha^1_{21}=0,\quad \alpha^1_{22}=0,\quad
    \alpha^2_{22}=0,\quad
    \alpha^1_{11}=\alpha^2_{21}-\alpha^2_{12}.
    \]
    
    The remaining free parameters are \(\alpha^2_{11}, \alpha^2_{12}, \alpha^2_{21}\).
    
    Now let \(g : B \to B\) be a linear map, written as  
    \[
    g(e_i) = \beta^1_i e_1 + \beta^2_i e_2, \quad \beta^k_i \in \Bbbk.
    \]
    
    The associated map \(\Psi g : \mathfrak{g} \otimes B \to \mathfrak{g} \otimes B\) is  
    \[
    \Psi g (x_i \otimes e_i) = x_i \otimes g(e_i).
    \]
    
    If \(\Psi f = \delta_{\text{Lie}}(\Psi g)\) is a coboundary, it is determined by the parameters \(g^1_2\) and \(2g^1_1 - g^2_2\). Consequently,  
    \[
    \dim_{\Bbbk} H_{\DL}^2(B, B) = 1.
    \]
\end{example}
\subsection*{Acknowledgments}
The author is grateful 
to  Kolesnikov P.S.
 and for discussions and useful comments.


\begin{thebibliography}{99}

\bibitem{chevalley-eilenberg}
C. Chevalley, S. Eilenberg,
\emph{Cohomology theory of Lie groups and Lie algebras},
Trans. Amer. Math. Soc. \textbf{63} (1948), 85--124.

\bibitem{LodayPirashvili}
J.-L. Loday, T. Pirashvili,
\emph{Universal enveloping algebras of Leibniz algebras and (co)homology},
Math. Ann. \textbf{296} (1993), no. 1, 139--158.

\bibitem{Chapoton2001}
F. Chapoton,
\emph{Un endofoncteur de la catégorie des opérades},
In: Dialgebras and related operads, 105--110,
Lecture Notes in Math., \textbf{1763}, Springer, Berlin, 2001.

\bibitem{I.S}
I. Alekseev and S.O. Ivanov, \emph{Higher Jacobi identities},
https://doi.org/10.48550/arXiv.1604.05281.


\bibitem{Loday}
J.-L. Loday,
\emph{Cup product for Leibniz cohomology and dual Leibniz algebras},
Math. Scand. \textbf{77} (1995), no. 2, 189--196.
\bibitem{LodayVallette}
J.-L. Loday, B. Vallette,
\emph{Algebraic Operads},
Grundlehren der mathematischen Wissenschaften, vol. 346,
Springer-Verlag, Berlin, 2012.
\bibitem{GubarevKolesnikov}
V. Gubarev, P. Kolesnikov,
\emph{On embedding of dendriform algebras into Rota-Baxter algebras},
J. Algebra \textbf{500} (2018), 153--170.

\bibitem{Yau}
Yau, D. (2007). Deformation of Dual Leibniz Algebra Morphisms. Communications in Algebra, 35(4), 1369–1378. 
\end{thebibliography}
\end{document}